\documentclass{amsart}

\usepackage{amssymb,amsmath,amsthm,verbatim}
\usepackage{bm} 
\usepackage{tikz} 
\usepackage{float} 
\usepackage{verbatim} 
\usetikzlibrary{calc}
\allowdisplaybreaks 
\newtheorem{theorem}{Theorem}
\newtheorem{theo}{Theorem}[section]
\newtheorem{lemma}[theo]{Lemma}
\newtheorem{cor}[theo]{Corollary}
\newtheorem{prop}[theorem]{Proposition}
\theoremstyle{definition}

\theoremstyle{remark}

\numberwithin{equation}{section}
\newcounter{tmp}

\newcommand{\ao}[1]{\tilde{#1}_{\omega}}
\newcommand{\aso}[1]{\tilde{#1}_{\sigma \omega}}
\newcommand{\as}[2]{\tilde{#1}_{\sigma^{#2} \omega}}
\newcommand{\bb}[1]{\tilde{b}_{#1}}
\newcommand{\cc}[1]{\tilde{c}_{#1}}
\newcommand{\dd}[1]{\tilde{d}_{#1}}
\newcommand{\ttt}[1]{\tilde{#1}}
\newcommand{\at}[2]{\tilde{#1}_{#2}}
\newcommand{\f}[1]{\mathcal{F}^\epsilon(#1)}
\newcommand{\ff}[2]{\mathcal{F}_{\sigma^{#2} \omega}^\epsilon(#1)}
\newcommand{\ft}[1]{\tilde{\mathcal{F}}^\delta(#1)}
\newcommand{\de}[1]{\mathcal{D}^\epsilon(#1)}
\newcommand{\dee}[2]{\mathcal{D}_{\sigma^{#2} \omega}^\epsilon(#1)}
\newcommand{\GG}{\mathcal{G}}
\newcommand{\GO}{\tilde{\mathcal{G}}}

\begin{document}

\author{Supun T. Samarakoon}
\title{On Growth of Generalized Grigorchuk's Overgroups}

\maketitle

\begin{abstract}
Grigorchuk's Overgroup $\tilde{\mathcal{G}}$, is a branch group of intermediate growth. It contains the first Grigorchuk's torsion group $\mathcal{G}$ of intermediate growth constructed in 1980, but also has elements of infinite order. It's growth is substantially greater than the growth of $\mathcal{G}$. The group $\mathcal{G}$, corresponding to the sequence $(012)^\infty = 012012 \hdots$, is a member of the family $\{ G_\omega | \omega \in \Omega = \{ 0, 1, 2 \}^\mathbb{N} \}$ consisting of groups of intermediate growth when sequence $\omega$ is not virtually constant. Following this construction we define the family $\{ \tilde{G}_\omega, \omega \in \Omega \}$ of generalized overgroups. Then $\tilde{\mathcal{G}} = \tilde{G}_{(012)^\infty}$ and $G_\omega$ is a subgroup of $\tilde{G}_\omega$ for each $\omega \in \Omega$. We prove, if $\omega$ is eventually constant, then $\tilde{G}_\omega$ is of polynomial growth and if $\omega$ is not eventually constant, then $\tilde{G}_\omega$ is of intermediate growth.
\end{abstract}

\section{Introduction}

The growth rate of groups are long studied area \cite{Sch55, Mil68, Gri91} and it was known that growth rates of groups can vary from polynomial growth through intermediate growth to exponential growth. First group of intermediate growth (the growth which is neither polynomial nor exponential), known as the first Grigorchuk's torsion group $\GG$, was constructed by Rostislav Grigorchuk in 1980 \cite{Gri80} as finitely generated infinite torsion group and later \cite{Gri83} it was shown that it has intermediate growth. The growth rate $\gamma_\GG(n)$ of $\GG$ was first shown to be bounded below by $e^{\sqrt{n}}$ and bounded above by $e^{n^\beta}$ where $\beta = \log_{32} {31} \approx 0.991$ \cite{Gri83,Gri84a}. In 1998, Laurent Bartholdi \cite{Bar98} and in 2001, Roman Muchnik and Igor Pak \cite{MP01} independently refined the upper bound to $\gamma_{\GG}(n) \preceq e^{n^\alpha},$ where $\alpha = \log{(2)} / \log{(2/\eta)} \approx 0.767$ and $\eta$ is the real root of the polynomial $x^3 + x^2 +x -2$. Recent work of Anna Erschler and Tianyi Zheng \cite{EZ18} showed $\gamma_\GG(n) \succeq e^{n^{(\alpha-\epsilon)}}$ for any positive $\epsilon$.

At the same time, in \cite{Gri83, Gri84a} (also see \cite{Gri85}) an uncountable family of groups $\{ G_\omega, \omega \in \Omega = \{ 0, 1, 2 \}^\mathbb{N} \}$, known as generalized Grigorchuk's groups were constructed. They consist of groups of intermediate growth when sequence $\omega$ is not virtually constant and of polynomial growth when sequence $\omega$ is virtually constant \cite{Gri84a}.

Since the construction of first Grigorchuk group, there was an expansion on the area of study and new groups of intermediate growth were introduced \cite{Gri84b,KP01,BE14,BGN15,Nek18}. The group $\GO$ known as the Grigorchuk's overgroup \cite{BG00} is an infinite finitely generated group of intermediate growth which shares many properties with first Grigorchuk's group \cite{BG02}. In contrast, the Grigorchuk's overgroup has an element which is non torsion \cite{BG00}. As a corollary to proposition \ref{lowerbound} and theorem \ref{upperboundomega0} of present article, growth bounds of the growth rate $\gamma_{\GO} (n)$ of $\GO$ satisfies, $\exp\left( \dfrac{n}{\log^{2+\epsilon}{n}} \right) \preceq \gamma_{\GO}(n) \preceq \exp\left(\dfrac{n \log(\log{n})}{\log{n}}\right)$ for any $\epsilon >0$.

First introduced technique for getting an upper bound for $\GG$ uses the strong contraction property \cite{Gri84a} (also known as sum contraction property), which says that there is a finite indexed subgroup $H$ of $\GG$ such that any element $g \in H$ can be uniquely decompose into some elements, whose sum of lengths in not larger than $C|g|+D$ where $0<C<1$ and $D$ are constants independent of $g$ \cite{Gri84a}. Later this technique was developed and many variants were introduced \cite{Bar03,Ers04,Fra17}. In 2004, Anna Erschler introduced a method for partial description of the Poisson boundary, to get a lower bound for certain class of groups of intermediate growth \cite{Ers04}. This idea was used to get the current known best lower bound for $\GG$ \cite{EZ18}. We will be using a version of strong contraction property in this text.

Following the construction in \cite{Gri84a}, we introduce an uncountable family $\{ \ao{G}, \omega \in \Omega \}$ called generalized Grigorchuk's overgroups (see section \ref{generalized}).

\begin{theorem} \label{tmain}
Let $\omega \in \Omega$. Then $\ao{G}$ is of polynomial growth if $\omega$ is virtually constant and $\ao{G}$ is of intermediate growth if $\omega$ is not virtually constant.
\end{theorem}

Let $\Omega_0^* \subset \Omega_0$ be the set of sequences $\omega$ such that there is an integer $M = M(\omega)$ with the property that for all $k \geq 1$, the set $\{\omega_k, \omega_{k+1}, \hdots \omega_{k+M-1} \}$ contains all three symbols and let $\Omega_1^* \subset \Omega_1$ be the set of sequences $\omega$ such that there is an integer $M = M(\omega)$ with the property that for all $k \geq 1$, the set $\{\omega_k, \omega_{k+1}, \hdots \omega_{k+M-1} \}$ contains at least two symbols.

\begin{theorem} \label{upperbound}
	Let $\omega \in \Omega_0^* \cup \Omega_1^*$. Then 
	\[\gamma_{\ao{G}}(n) \preceq \exp\left(\dfrac{n \log(\log{n})}{\log{n}}\right). \]
\end{theorem}

We prove theorem \ref{tmain} in section \ref{growthproof} and theorem \ref{upperbound} in section \ref{growthbounds}. Also we provide a lower bound for the growth of all the groups $\ao{G}$ by a function of type $\exp{\left(\dfrac{n}{\log^{2+\epsilon}(n)}\right)}$ for arbitrary $\epsilon >0$ (see proposition \ref{lowerbound}).

\section{preliminaries}\label{prelim}

First we introduce some notations. Let $\Omega = \{ 0, 1, 2 \}^\mathbb{N}$ and let $\Omega_0, \Omega_1, \Omega_2, \Omega_{1,2}$ be subsets of $\Omega$, where $\Omega_0$ is the set consisting of all sequences containing $0, 1$ and $2$ infinitely often, $\Omega_2$ is the set consisting of all eventually constant sequences, $\Omega_1 = \Omega - (\Omega_0 \cup \Omega_2)$, and $\Omega_{1,2}$ is the set consisting sequences containing at most two symbols. Let $\sigma : \Omega \rightarrow \Omega$ be the left shift. i.e. $(\sigma\omega)_n = \omega_{n+1}$.

\subsection{Generalized Grigorchuk's Groups $\bm{G_\omega}$ and Generalized Grigorchuk's Overgroups $\bm{\ao{G}}$}\label{generalized}

Consider the labeled  binary rooted tree $T_2$ [see figure ~\ref{fig:tree}]. For each vertex $v$, let $I$ be the trivial action on $v$ and let $P$ be the action of interchanging the vertices $v0,v1$ and acting trivially on these two vertices. We identify an infinite sequence $\{a_n\}$ of $P,I$ with the element $g \in Aut(T_2)$ such that $g \cdot (1^{n-1}0) = a_n$. We define $a$ to be the element acting on the root as $P$ and trivially on other vertices and $x$ to be the element $(P,P, \hdots)$.

\tikzset{
	s/.style={circle,draw,inner sep=0,fill=black},
	l/.style={circle,draw,inner sep=0,fill=black,xshift=10},
	r/.style={circle,draw,inner sep=0,fill=black,xshift=-10},
}
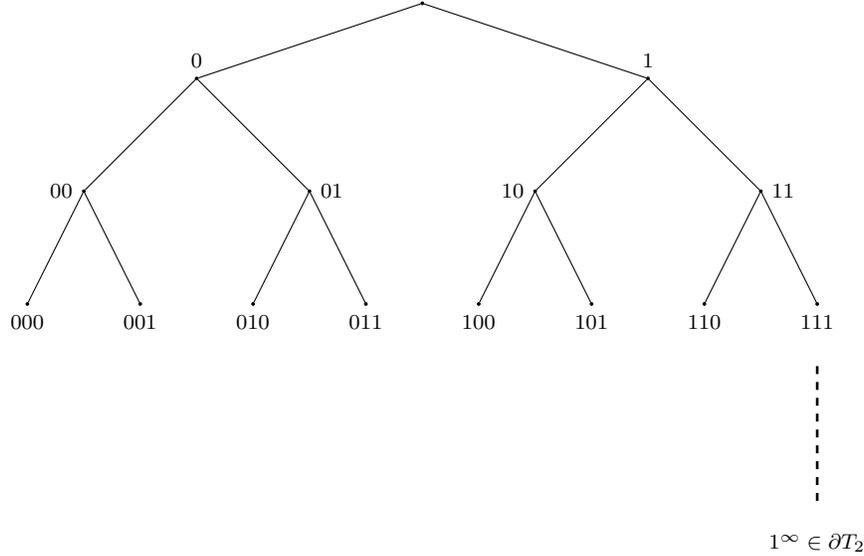
\begin{figure}[t!]
	\centering
	\begin{tikzpicture}[scale=1,font=\footnotesize]
	\tikzstyle{level 1}=[level distance=10mm,sibling distance=60mm]
	\tikzstyle{level 2}=[level distance=15mm,sibling distance=30mm]
	\tikzstyle{level 3}=[level distance=15mm,sibling distance=15mm]
	\tikzstyle{level 4}=[level distance=10mm,sibling distance=10mm]
	\node(-1)[s,label=above:{}]{}
	child{node(0)[s,label=above:{$0$}]{}
		child{node(00)[s,label=left:{${00}$}]{}
			child{node(100)[s,label=below:{${000}$}]{}}
			child{node(101)[s,label=below:{${001}$}]{}}
		}
		child{node(01)[s,label=right:{${01}$}]{}
			child{node(110)[s,label=below:{${010}$}]{}}
			child{node(111)[s,label=below:{${011}$}]{}}
		}
	}
	child{node(1)[s,label=above:{$1$}]{}
		child{node(10)[s,label=left:{${10}$}]{}
			child{node(100)[s,label=below:{${100}$}]{}}
			child{node(101)[s,label=below:{${101}$}]{}}
		}
		child{node(11)[s,label=right:{${11}$}]{}
			child{node(110)[s,label=below:{${110}$}]{}}
			child{node(111)[s,label=below    :{${111}$}]{}}
		}
	};
	\node(1111)[yshift=-20] at ($(111)$) {};
	\node(11111)[yshift=-80,label=below:{$1^\infty \in \partial T_2$}] at ($(111)$) {};
	\draw[dashed,line width=1 pt] (1111) to (11111);
	\end{tikzpicture}
	\caption{Labeled binary rooted tree $T_2$}
\label{fig:tree}
\end{figure}

For $\omega \in \Omega$, identify elements $b_\omega, c_\omega, d_\omega$ with sequences $\{b_n\}, \{c_n\}, \{d_n\}$, respectively, where
\\
$ b_n = \begin{cases} 
P & \omega_n = 0 \text{ or } 1 \\
I & \omega_n = 2 
\end{cases} \quad , \quad
c_n = \begin{cases} 
P & \omega_n = 0 \text{ or } 2 \\
I & \omega_n = 1
\end{cases} \quad \text{ and } \quad
d_n = \begin{cases} 
P & \omega_n = 1 \text{ or } 2 \\
I & \omega_n = 0
\end{cases}
$.
\\
Further define $\ao{b} := xb_\omega, \; \ao{c} := xc_\omega$ and $\ao{d} := xd_\omega$. Note that all these elements are involutions and all except $a$ commute with each other. The generalized Grigorchuk's group $G_\omega$ is the group generated by elements $a,b_\omega, c_\omega, d_\omega$ and the generalized overgroup $\ao{G}$ is the group generated by $a,b_\omega, c_\omega, d_\omega, x$. $G_\omega \subset \ao{G}$ and it is useful to view $\ao{G}$ as the group generated by elements $a,b_\omega, c_\omega, d_\omega, x, \ao{b}, \ao{c}, \ao{d}$, where a typical element $g \in \ao{G}$ can be represented in reduced form $(a)*a*a \hdots a*a*(a)$ where first and last $a$ can be omitted and $*$s represent generators other than $a$, using simple contractions \ref{e1}, which will be introduced later.

Denote $\ao{H} := \ao{H}^{(1)} := Stab_{\ao{G}}(1)$ and $g \in \ao{H}$ if and only if $g$ has even number of $a$'s. There is a natural embedding $\ao{\psi}$ from $\ao{H}$ into $\aso{G}\times \aso{G}$ given by $\ao{\psi}(g) = (g|_0,g|_1)$, where $g|_v$ is the restricted action on rooted tree with root $v$, for $v=0,1$. We will denote $\ao{\psi}$ by $\ttt{\psi}$ if $\omega$ is understood.

Volume growth function of group $G$ with finite generating set $S$, denoted by $\gamma_{G,S}$ is defined by $\gamma_{G,S}(n) =$ number of elements $g \in G$ which can be written as a product of $n$ or less number of generators of $S$. There is an order relation $\preceq$ for growth functions defined by $f \preceq g$ if and only if there are constants $A$ and $B$ such that $f(n) \leq Ag(Bn)$ for all $n$. We define an equivalence relation $\simeq$ by, $f \simeq g$ if and only if $f \preceq g$ and $g \preceq f$. The equivalence classes are know as growth rates, which are independent of generating set. Growth rate can be polynomial, exponential or intermediate. Growth above polynomial is called super polynomial and growth below exponential is called subexponential.

The growth exponent $\lambda_G$ of group $G$ is given by $\lambda_G = \lim_n [\gamma_G(n)]^{1/n}$. It is known that $\lambda_G >1 \iff G$ has exponential growth \cite{Gri84a}. We will be using $\ao{\gamma}, \ao{\lambda}$ in this text to denote $\gamma_{\ao{G},\ao{S}}, \lambda_{\ao{G}}$, where $\ao{S} = \{ a,b_\omega, c_\omega, d_\omega, x, \ao{b}, \ao{c}, \ao{d} \}$.

\section{Growth of $\bm{\ao{G}}$}\label{growthproof}


\begin{prop} \label{1}
Let $\omega \in \Omega_1 \cup \Omega_2$. Then $\ao{G}$ has subexponential growth.
\end{prop}

\begin{lemma} \label{2}
A non-decreasing semi-multiplicative function $\gamma(n)$ with argument a natural number, can be extended to a non-decreasing semi-multiplicative function $\gamma(x)$, with argument a non-negative real number.
\end{lemma}

\begin{lemma}\label{3}
For any $\omega \in \Omega, \ao{\lambda} \leq \aso{\lambda}$
\begin{proof}
Denote $\ao{B}(n) = \{ g \in \ao{G} : |g| \leq n \}$ and $ \ao{H}(n) = \ao{H} \cap \ao{B}(n)$. Any element $g \in \ao{B}(n)$ is either in $\ao{H}$ or is of the form $ g = ag'$, where $g' \in  \ao{H}$ and $|g'| \leq |g| + 1 \leq n+1$. Thus,
\[ \ao{\gamma}(n) = |\ao{B}(n)| \leq |\ao{H}(n)| + |\ao{H}(n+1)| \leq 2|\ao{H}(n+1)|. \]
For each $g \in \ao{H}$ there are unique $g_l, g_r \in \aso{G}$ such that $|g_l|,|g_r| \leq \frac{|g|+1}{2}$ and $\ao{\psi}(g) = (g_l, g_r)$. Thus,
\[ |\ao{H}(n)| \leq |\aso{B}(\frac{n+1}{2})|^2 = [ \aso{\gamma}(\frac{n+1}{2}) ]^2. \]
Therefore,
\[ \ao{\gamma}(n) \leq 2 [ \aso{\gamma}(\frac{n+2	}{2}) ]^2. \]
Consequently,
\[ \ao{\lambda} = \lim_n [\ao{\gamma}(n)]^{1/n} \leq \lim_n [2 [ \aso{\gamma}(\frac{n+2	}{2}) ]^2]^{1/n} = \lim_n  [ \aso{\gamma}(\frac{n+2	}{2})]^{2/n} = \aso{\lambda}. \]
\end{proof}
\end{lemma}

\begin{lemma} \label{4}
For any $\omega \in \Omega_{1,2}$, $\ao{G} = G_{\omega}$.
\begin{proof}
First note that $ x \in G_{\omega} \implies xb_\omega , xc_\omega , xd_\omega \in G_\omega \implies \ao{b}, \ao{c}, \ao{d} \in G_\omega \implies \ao{G} \subset G_{\omega} \implies \ao{G} = G_{\omega}$. To prove lemma, we only need to show that $x \in G_{\omega}$. For definiteness we may assume $\omega$ consists only of symbols $0,1$. Since the first entry of $0$ and $1$ is $P, b_\omega = (P, P, P, ... ) = x$. Therefore $ x \in G_{\omega}$ and thus the result is true.
\end{proof}
\end{lemma}

\begin{proof}[{Proof of Proposition \ref{1}}]
Let $\omega \in \Omega_1 \cup \Omega_2$. Then there exists $N \in \mathbb{N}$ such that $\sigma^N \omega$ consists only of at most two symbols. Then by lemma \ref{4}, $\as{G}{N} = G_{\sigma^N \omega}$. Therefore $\as{\lambda}{N} = \lambda_{\sigma^N \omega}$. For any $\omega$, $ G_{\omega}$ is of intermediate growth if $\omega \in \Omega_1$ and of polynomial growth if $\omega \in \Omega_2$ \cite{Gri84a}. Thus $\lambda_{\sigma^N \omega} = 1$. So by lemma \ref{3}, $\ao{\lambda} \leq \as{\lambda}{N} = 1$. Thus $\ao{G}$ is of subexponential growth.
\end{proof}


\begin{prop} \label{5}
Let $\omega \in \Omega_1$. Then $\ao{G}$ has intermediate growth.
\begin{proof}
By theorem \ref{1}, $\ao{G}$ is of subexponential growth. Since $G_{\omega} \subset \ao{G}$ and $G_{\omega}$ is of super-polynomial growth \cite{Gri84a}, $\ao{G}$ is of super-polynomial growth. Hence $\ao{G}$ is of intermediate growth.
\end{proof}
\end{prop}

\begin{prop} \label{6}
Let $\omega \in \Omega_2$. Then $\ao{G}$ has polynomial growth.
\begin{proof}
Since $\omega \in \Omega_2$, there is a natural number $N$ such that $\omega = \{ \omega_n \}, \omega_n = \omega_N $ for all $N \geq n$. Then $\as{G}{N-1} = \left\langle a, \at{b}{N}, \at{c}{N}, \at{d}{N} \right\rangle = \left\langle a, x \right\rangle \cong \mathbb{D}_\infty$, the infinite Dihedral group. Let $\mathbb{G}$ be the group acting on binary tree $T$ containing all the elements $g \in Aut(T)$ such that $g$ restricted to any sub-tree tooted at a vertex in level $N-1$ is in $\left\langle a, x \right\rangle$. Then $\ao{G} \subset \mathbb{G}$. Let $\mathbb{G}_0$ be the subgroup of $\mathbb{G}$ containing elements acting trivially on the first $N-1$ levels of the binary tree $T$. Note that $ \mathbb{G}_0 \triangleleft \mathbb{G}$ and $[\mathbb{G}:\mathbb{G}_0] \leq 2^{2^{N}-1}$. But $\mathbb{G}_0 \cong \left\langle a, x \right\rangle ^{2^{N-1}} \cong \mathbb{D}_\infty^{2^{N-1}}$. Thus $\mathbb{G}_0$ is virtually abelian and thus of polynomial growth. Since $[\mathbb{G}:\mathbb{G}_0] < \infty, \mathbb{G}$ is of polynomial growth. $\ao{G} \subset \mathbb{G}$ implies that $\ao{G}$ is of polynomial growth.
\end{proof}
\end{prop}


\begin{prop} \label{7}
Let $\omega \in \Omega_0$. Then $\ao{G}$ has intermediate growth.
\end{prop}

We will, now on, consider the generating set of $\ao{G}$ to be $\ao{S} = \{ a, b_\omega, c_\omega, d_\omega, \ao{b}, \ao{c}, \ao{d}, x \}$. Then we have the following relations called simple contractions;
\begin{gather}
a^2 = x^2 = b_\omega^2 = c_\omega^2 = d_\omega^2 = \ao{b}^2 = \ao{c}^2 = \ao{d}^2 = 1 \nonumber
\\
b_\omega c_\omega = c_\omega b_\omega = d_\omega, \quad
c_\omega d_\omega = d_\omega c_\omega = b_\omega, \quad
d_\omega b_\omega = b_\omega d_\omega = c_\omega \nonumber
\\
\ao{b} \ao{c} = \ao{c} \ao{b} = d_\omega, \quad
\ao{c} \ao{d} = \ao{d} \ao{c} = b_\omega, \quad
\ao{d} \ao{b} = \ao{b} \ao{d} = c_\omega \nonumber
\\
b_\omega \ao{c} = \ao{c} b_\omega = \ao{d}, \quad
c_\omega \ao{d} = \ao{d} c_\omega = \ao{b}, \quad
d_\omega \ao{b} = \ao{b} d_\omega = \ao{c} \label{e1}
\\
\ao{b} c_\omega = c_\omega \ao{b} = \ao{d}, \quad
\ao{c} d_\omega = d_\omega \ao{c} = \ao{b}, \quad
\ao{d} b_\omega = b_\omega \ao{d} = \ao{c} \nonumber 
\\
b_\omega \ao{b} = \ao{b} b_\omega = c_\omega \ao{c} = \ao{c} c_\omega = d_\omega \ao{d} = \ao{d} d_\omega = x \nonumber
\\
b_\omega x = x b_\omega = \ao{b}, \quad
c_\omega x = x c_\omega = \ao{c}, \quad
d_\omega x = x d_\omega = \ao{d} \nonumber
\\
\ao{b} x = x \ao{b} = b_\omega, \quad
\ao{c} x = x \ao{c} = c_\omega, \quad
\ao{d} x = x \ao{d} = d_\omega \nonumber
\end{gather}

Any word in the alphabet $\ao{S}$ is called reduced if it is of the form $(a) \ast a \ast \hdots \ast a \ast (a)$ where first and last $a$ can be omitted. Any word can be reduced using simple contractions. The length of a word $W$ denoted by $|W|$ is the number of letters in $W$ and let $|W|_\ast$ denote the number of $\ast$'s in $|W|$ for $\ast \in \ao{S}$. For any element $g \in \ao{G}$ the length of $g$ denoted by $|g|$ is defined by,
\[ |g| = min \{ |W| : g = W \, \textrm{in} \, \ao{G}, \, W \, \textrm{is reduced} \}. \]
A reduced word $W$ satisfying $ g = W$ in $\ao{G}$ and $|g| = |W|$ is called a minimal representation of $g$. For any $\epsilon > 0$ define $\f{n}$ to be the set of length $n$ elements $g$ in $\ao{G}$ with every minimal representation $W$ of $g$ of alphabet $\ao{S}$ satisfies at least one of the inequalities
\[ |W|_\ast > (1/2 - \epsilon)n; \quad \textrm{where } \, \ast \in \ao{S}-\{a\}. \]
Let $\mathcal{D}^\epsilon(n)$ be the set of length $n$ elements $g$ in $\ao{G}$ having at least one minimal representation $W$ of $g$ of alphabet $\ao{S}$ satisfying the inequalities
\[ |W|_\ast \leq (1/2 - \epsilon)n; \quad \textrm{for all } \, \ast \in \ao{S}-\{a\}. \]
For any $\delta > 0$ define $\ft{n'}$ to be the the set of words in the alphabet $\ao{S}-\{a\}$ of length $n'$ such that every $W' \in \ft{n'}$ satisfies at least one inequality
\[ |W'|_\ast > (1 - \delta)n'; \quad \textrm{where } \, \ast \in \ao{S}-\{a\}. \]

\begin{lemma} \label{8}
For any $W \in \f{n}$ there is $W' \in \ft{n'}$ obtained by deleting $a$'s in $W$ where
\begin{gather*}
\frac{n-1}{2} \leq n' \leq \frac{n+1}{2} \\
\delta = 2 \epsilon + \frac{3}{n-1}
\end{gather*}
\begin{proof}
Since there are almost half of letter $a$'s and thus we get $\frac{n-1}{2} \leq n' \leq \frac{n+1}{2}$. Also $|W'|_\ast = |W|_\ast >(1/2 - \epsilon)n \geq (1/2 - \epsilon)(2n' - 1) = (1 - 2\epsilon - \frac{(1-2\epsilon)}{2n'})n' > (1 - 2\epsilon - \frac{3}{n-1})n'$.
\end{proof}
\end{lemma}

\begin{lemma} \label{9}
If $\delta < 6/e$, then $\overline{\lim\limits_{n}} |\ft{k}|^{1/k} \leq (1 - \delta)^{-1} (\delta / 6)^{-\delta}$.
\begin{proof}
Since any of seven letters $\{ b, c, d, \ttt{b}, \ttt{c}, \ttt{d}, x \}$ can enter into the word $W \in \ft{k}$ with frequency $> 1 - \delta$, we have,
\begin{align*}
|\ft{k}| \leq & 7 + 7 \sum_{x = 0}^{[\delta k]} {\sum_{\sum {i_j} = [\delta k] -x} {\frac{k!}{(k - [\delta k] + x)! i_1! \hdots i_6!}}} \\
 \leq & 7 + 7 \sum_{x = 0}^{[\delta k]} {[\delta k] - x +5 \choose 5} \frac{k!}{(k - [\delta k] + x)! \left( \frac{(\delta k - x)_*}{6} \right)!^6} \\
& \quad \quad \quad \text{; where } (\delta k - x)_* := 6 \left[ \frac{[\delta k - x]}{6} \right] \\
 \leq & 7 + 7 {[\delta k] +5 \choose 5} \sum_{x = 0}^{[\delta k]} \frac{k!}{(k - [\delta k] + x)! \left( \frac{(\delta k - x)_*}{6} \right)!^6} \\
 \leq & ([\delta k] + 5)^5 \sum_{x = 0}^{[\delta k]} \frac{k!}{(k - [\delta k] + x)! \left( \frac{(\delta k - x)_*}{6} \right)!^6} \\
 \leq & ([\delta k] + 5)^5 \sum_{x = 0}^{[\delta k]} \frac{c\sqrt{k} k^k e^{-k} e^{(k - [\delta k] + x)} e^{(\delta k - x)_*}}{(k - [\delta k] + x)^{(k - [\delta k] + x)} \left( \frac{(\delta k - x)_*}{6} \right)^{(\delta k - x)_*}} \\
& \quad \quad \quad \text{; by Stirling's formula } \frac{n^n}{e^n} \leq n! \leq c \sqrt{n} \frac{n^n}{e^n} \\
 \leq & c([\delta k] + 5)^5 \sum_{x = 0}^{[\delta k]} \frac{\sqrt{k} k^{([\delta k] - x) - (\delta k - x)_*} e^{(\delta k - x)_* - ([\delta k] - x)}}{\left(1 - \frac{[\delta k]}{k} + \frac{x}{k}\right)^{(k - [\delta k] + x)} \left( \frac{(\delta k - x)_*}{6k} \right)^{(\delta k - x)_*}} \\
& \quad \quad \quad \text{; since } 0 \leq ([\delta k] - x) - (\delta k - x)_* \leq 6 \\
 \leq & c([\delta k] + 5)^5 \sum_{x = 0}^{[\delta k]} \frac{\sqrt{k} k^{6}}{\left(1 - \frac{[\delta k]}{k} + \frac{x}{k}\right)^{(k - [\delta k] + x)} \left( \frac{(\delta k - x)_*}{6k} \right)^{(\delta k - x)_*}} \\
 \leq & c  k^6 \sqrt{k} (1 - \delta)^{-k} \sum_{x = 0}^{[\delta k]} {\left( \frac{(\delta k - x)_*}{6k} \right)^{-(\delta k - x)_*}}.
\end{align*}
Since the function $x^{-x}, x > 0$, is increasing in the interval $(0, e^{-1})$ and since $\delta / 6 < e^{-1}$, 
\[ \left( \frac{(\delta k - x)_*}{6k} \right)^{-\left( \frac{(\delta k - x)_*}{6k} \right)} \leq \left( \frac{\delta}{6} \right)^{-\left( \frac{\delta}{6} \right)} \]
and therefore,
\begin{equation*}
|\ft{k}| \leq c  k^6 \sqrt{k} (1 - \delta)^{-k} ([\delta k] + 1) \left( \frac{\delta}{6} \right)^{-\left( \frac{\delta}{6} \right) 6k}.
\end{equation*}
Hence,
\[ \overline{\lim\limits_{n}} |\ft{k}|^{1/k} \leq (1 - \delta)^{-1} (\delta / 6)^{-\delta} \]
\end{proof}
\end{lemma}

\begin{cor} \label{10}
$\overline{\lim\limits_{n}} |\f{n}|^{1/n} \leq (1 - 2\epsilon)^{-1/2} (\epsilon / 3)^{-\epsilon}$.
\begin{proof}
If $n$ is even, then at most two words in $\f{n}$ gives the same word in $\ft{n/2}$. So,
\[ |\f{n}| \leq 2 |\ft{n/2}|. \]
If $n$ is odd, then each word in $\f{n}$ gives a unique word in $\ft{(n-1)/2}$ or $\ft{(n+1)/2}$ and so,
\[ |\f{n}| \leq |\ft{(n-1)/2}| + |\ft{(n+1)/2}|.\]
Note that,
\[ \overline{\lim\limits_{n}} |\ft{n/2}|^{1/n} \leq \lim_n \left( (1 - \delta)^{-1} (\delta / 6)^{-\delta} \right)^{1/2} \]
\[ \overline{\lim\limits_{n}} |\ft{(n-1)/2}|^{1/n} \leq \lim_n \left( (1 - \delta)^{-1} (\delta / 6)^{-\delta} \right)^{1/2} \]
\[ \overline{\lim\limits_{n}} |\ft{(n+1)/2}|^{1/n} \leq \lim_n \left( (1 - \delta)^{-1} (\delta / 6)^{-\delta} \right)^{1/2} \]
and thus,
\[ \overline{\lim\limits_{n}} |\f{n}|^{1/n} \leq \lim_n \left( (1 - \delta)^{-1} (\delta / 6)^{-\delta} \right)^{1/2}. \]
Since $\delta = 2 \epsilon + \frac{3}{n-1}$, $\lim_n \delta = 2 \epsilon$ and therefore,
\[ \lim_n \left( (1 - \delta)^{-1} (\delta / 6)^{-\delta} \right)^{-1/2} = (1 - 2\epsilon)^{-1/2} (\epsilon / 3)^{-\epsilon}. \]
Hence we get the desired result.
\end{proof}
\end{cor}

Let $\ao{H}^{(s)} := Stab_{\ao{G}}(s)$ and $\as{H}{n}^{(s)} := Stab_{\as{G}{n}}(s)$. For each $s$, denote by $ a, b_s, c_s, d_s, \bb{s}, \cc{s}, \dd{s}, x $ the canonical generators of $\as{G}{s}$. So $s=0$ gives the generators of $\ao{G}$. Using the map $\psi$, we get the following;\\
\begin{align} \label{e2}
\omega_s = 0 \implies & b_{s-1} = (a, b_s) \quad c_{s-1} = (a, c_s) \quad d_{s-1} = (1, d_s) \quad x = (a, x) \\
& \bb{s-1} = (1, \bb{s}) \quad \cc{s-1} = (1, \cc{s}) \quad \dd{s-1} = (a, \dd{s}) \nonumber \\ \nonumber \\
\omega_s = 1 \implies & b_{s-1} = (a, b_s) \quad c_{s-1} = (1, c_s) \quad d_{s-1} = (a, d_s) \quad x = (a, x) \nonumber \\
& \bb{s-1} = (1, \bb{s}) \quad \cc{s-1} = (a, \cc{s}) \quad \dd{s-1} = (1, \dd{s}) \nonumber \\ \nonumber \\
\omega_s = 2 \implies & b_{s-1} = (1, b_s) \quad c_{s-1} = (a, c_s) \quad d_{s-1} = (a, d_s) \quad x = (a,x) \nonumber \\
& \bb{s-1} = (a, \bb{s}) \quad \cc{s-1} = (1, \cc{s}) \quad \dd{s-1} = (1, \dd{s}) \nonumber
\end{align}
\\
Let $W$ represent a word in $\ao{H}^{(s)}$. Then there are $\tilde{W}_0, \tilde{W}_1$ such that $ W = (\tilde{W}_0, \tilde{W}_1)$ using substitutions in \ref{e2}. Let $W_0, W_1$ be obtained by doing simple contractions on $\tilde{W}_0, \tilde{W}_1$. Let $\alpha_1$ denote the number of such simple contractions. So $W_0,W_1$ represent words in $\as{H}{1}^{(s-1)}$. Now there are $\tilde{W}_{00}, \tilde{W}_{01}, \tilde{W}_{10}, \tilde{W}_{11}$ such that $ W_0 = (\tilde{W}_{00}, \tilde{W}_{01}), W_1 = (\tilde{W}_{10}, \tilde{W}_{11})$ using substitutions in \ref{e2}. Let $W_{00}, W_{01}$, $W_{10}, W_{11}$ be obtained by doing simple contractions on $ \tilde{W}_{00}, \tilde{W}_{01}, \tilde{W}_{10}$, $\tilde{W}_{11}$. Let $\alpha_2$ denote the number of such simple contractions. So $W_{00}, W_{01}, W_{10}, W_{11}$ represent words in $\as{H}{2}^{(s-2)}$. Proceeding this manner we get $\{ W_{i_1 i_2 \hdots i_s} \}_{i_i \in \{0,1\}}$ representing words in $\as{H}{s}^{(s-s)} = \as{G}{s}$. Denote by $\alpha_s$ the number of simple contractions done to obtain $\{ W_{i_1 i_2 \hdots i_s} \}_{i_i \in \{0,1\}}$ from $\{ \tilde{W}_{i_1 i_2 \hdots i_s} \}_{i_i \in \{0,1\}}$.
Let $x_0 := |W|_{d_0} + |W|_{\bb{0}} + |W|_{\cc{0}}, y_0 := |W|_{c_0} + |W|_{\bb{0}} + |W|_{\dd{0}}$ and $z_0 := |W|_{b_0} + |W|_{\cc{0}} + |W|_{\dd{0}}$. Also for $j = 1, 2, \hdots s$, let 
\[x_j = \sum {\left( |W_{i_1 i_2 \hdots i_j}|_{d_j} + |W_{i_1 i_2 \hdots i_j}|_{\bb{j}} + |W_{i_1 i_2 \hdots i_j}|_{\cc{j}} \right)} \]
\[y_j = \sum {\left( |W_{i_1 i_2 \hdots i_j}|_{c_j} + |W_{i_1 i_2 \hdots i_j}|_{\bb{j}} + |W_{i_1 i_2 \hdots i_j}|_{\dd{j}} \right)} \]
\[z_j = \sum {\left( |W_{i_1 i_2 \hdots i_j}|_{b_j} + |W_{i_1 i_2 \hdots i_j}|_{\cc{j}} + |W_{i_1 i_2 \hdots i_j}|_{\dd{j}} \right)}. \]

\begin{lemma} \label{11}
Let $\epsilon > 0, n_\epsilon \in \mathbb{N}$ such that $n_\epsilon \epsilon > 5/2$. Let $n \geq n_\epsilon$. Let $s \in \mathbb{N}$ such that $\omega_s$ is the first time that the third symbol appears in $\omega$. Let $W \in \de{n}$ represent a word in $\ao{H}^{(s)}$. Then,
\[ \sum_{i_1, i_2, \hdots , i_s} {|W_{i_1 i_2 \hdots i_s}|} \leq \left( 1-\frac{\epsilon}{5} \right)n + 2^s - 1.\]
\begin{proof}
For definiteness, suppose $\omega_1 = \hdots = \omega_{t-1} = 0$, $\omega_t = 1$, $\omega_m \neq 2$ for every $m <s$ and $\omega_s = 2$. First note that each simple contraction decreases $y_i,z_i$ by at most 2. Thus,
\[y_{t-1} \geq y_0 - 2(\alpha_1 + \alpha_2 + \hdots + \alpha_{t-1}) \geq y_0 - 2\sum_1^{s-1} \alpha_i \]
\[z_{s-1} \geq z_0 - 2\sum_1^{s-1} \alpha_i. \]
Also note that,
\[ \sum_{i_1, i_2, \hdots , i_s} {|W_{i_1 i_2 \hdots i_s}|} \leq n + 2^s - 1 - x_0 - y_{t-1} - z_{s-1} - \sum_1^{s-1} \alpha_i.\]
Now let us show that $ x_0 + y_{t-1} + z_{s-1} + \sum_1^{s-1} \alpha_i > n\epsilon /5$. To the contrary assume $ x_0 + y_{t-1} + z_{s-1} + \sum_1^{s-1} \alpha_i \leq n\epsilon /5$. Therefore, Thus $\sum_1^{s-1} \alpha_i \leq n\epsilon /5$. Therefore,
\begin{align*}
x_0 + y_0 + z_0 & \leq x_0 + \left( y_{t-1} + 2\sum_1^{s-1}{\alpha_i} \right) + \left( z_{s-1} + 2\sum_1^{s-1}{\alpha_i} \right) \\
& \leq \left( x_0 + y_{t-1} + z_{s-1} + \sum_1^{s-1}{\alpha_i} \right) + 3 \left( \sum_1^{s-1}{\alpha_i} \right) \\
& \leq \frac{4}{5}n\epsilon
\end{align*}
But $n = |W| \leq |W|_a + |W|_x + x_0 + y_0 + z_0 \leq \frac{n+1}{2} + \frac{n}{2} -n\epsilon + \frac{4}{5}n\epsilon$, thus $ n\epsilon \leq 5/2$, which is a contradiction. So $ x_0 + y_{t-1} + z_{s-1} + \sum_1^{s-1} \alpha_i > n\epsilon /5$. Therefore,
\[ \sum_{i_1, i_2, \hdots , i_s} {|W_{i_1 i_2 \hdots i_s}|} \leq \left( 1-\frac{\epsilon}{5} \right)n + 2^s - 1.\]
\end{proof}
\end{lemma}

\begin{proof}[{Proof of proposition \ref{7}}]
Take a fixed $\epsilon > 0$. If for at least one $k$ there exist infinite set $N_0 \subset \mathbb{N}$ such that $\forall n \in N_0$
\begin{equation} \label{e2.1}
|\ff{n}{k}| \geq |\dee{n}{k}|
\end{equation}
then,
\begin{align} \label{e3}
\ao{\lambda} & = \as{\lambda}{k} \nonumber \\
& = \lim_{n} |\as{\gamma}{k}(n)|^{1/n} \nonumber \\
& = \lim_{n \in N_0} |\as{\gamma}{k}(n)|^{1/n} \nonumber \\
& = \lim_{n \in N_0} [|\ff{n}{k}| + |\dee{n}{k}|]^{1/n} \nonumber \\
& \leq \overline{\lim_{n \in N_0}} [2|\ff{n}{k}|]^{1/n} \nonumber \\
& \leq \overline{\lim_{n}} [2|\ff{n}{k}|]^{1/n} \nonumber \\
& = \overline{\lim_{n}} |\ff{n}{k}|^{1/n} \nonumber \\
\ao{\lambda} & \leq (1 - 2\epsilon)^{-1/2} (\epsilon / 3)^{-\epsilon}.
\end{align}
Now suppose that for every $k \in \mathbb{N}$ there exists an $N(k)$ such that for all $n > N(k)$
\begin{equation} \label{e4}
 |\ff{n}{k}| < |\dee{n}{k}|.
\end{equation}
As before let $\ao{H}^{(s)}(n) := \ao{B}(n) \cap \ao{H}^{(s)}$ and $\as{H}{k}^{(s)}(n) := \as{B}{k}(n) \cap \as{H}{k}^{(s)}$. Let $\omega = \omega_1 \hdots \omega_{s_1} \omega_{s_1 + 1} \hdots \omega_{s_1 + s_2} \omega_{s_1 + s_2 + 1} \hdots \omega_{s_1 + s_2 + s_2} \hdots$ where $s_1$ is the first time third symbol appears in $\omega$, $s_2$ is the first time third symbol appears in $\sigma^{s_1} \omega$, and so on.

Since $[\ao{G}:\ao{H}^{(s_1)}] \leq (2^{s_1})! =: K_1$, there is a fixed Schreier system of representatives of the right cosets of $\ao{G}$ modulo $\ao{H}^{(s_1)}$ with each Schreier representative is of length less than $K_1$. So for any $g \in \ao{B}(n)$, there are $h \in \ao{H}^{(s_1)}$, $l$ a Schreier representative such that $g = hl$ and since $|l| \leq K_1$, we have $|h| \leq n+K_1$. Therefore,
\begin{equation} \label{e5}
|\ao{B}(n)| \leq K_1 |\ao{H}^{(s)}(n+K_1)| .
\end{equation}
By \ref{e4} we get,
\begin{equation*}
|\ao{H}^{(s)}(n+K_1)| \leq 3 \sum_{k=1}^{n+K_1} |\ao{H}^{(s_1)}(n+K_1) \cap \de{k}|
\end{equation*}
and by lemma \ref{11},
\begin{equation} \label{e6}
|\ao{H}^{(s)}(n+K_1)| \leq 3 \sum_{i_1, \hdots, i_{2^{s_1}}} |\as{B}{s_1}(i_1)| \hdots |\as{B}{s_1}(i_{2^{s_1}})|
\end{equation}
where $\sum_{j =1}^{2^{s_1}} i_j \leq \left( 1-\frac{\epsilon}{5} \right)(n+K_1) + 2^{s_1} - 1$.

The growth index $\as{\lambda}{s_1}$ of the group $\as{G}{s_1}$ is defined by the relation
\[ \as{\lambda}{s_1} = \lim_{i} |\as{B}{s_1}(i)|^{1/i} \]
and therefore for each $\delta >0$ there exists an $I = I(\delta)$ such that for $i \geq I$,
\[ |\as{B}{s_1-1}(i)| \leq (\as{\lambda}{s_1} + \delta)^i. \]
Thus for all $i$,
\[ |\as{B}{s_1-1}(i)| \leq |\as{B}{s_1-1}(I)| (\as{\lambda}{s_1} + \delta)^i \]
which implies
\begin{align} \label{e7}
|\as{B}{s_1}(i_1)| \hdots |\as{B}{s_1}(i_{2^{s_1}})| & \leq |\as{B}{s_1-1}(I)|^{2^{s_1}} (\as{\lambda}{s_1} + \delta)^{\sum_{j =1}^{2^{s_1}} i_j} \nonumber \\
 & \leq |\as{B}{s_1-1}(I)|^{2^{s_1}} (\as{\lambda}{s_1} + \delta)^{\left( 1-\frac{\epsilon}{5} \right)(n+K_1) + 2^{s_1} - 1}.
\end{align}
Number of summands in the right hand side of \ref{e6} is,
\begin{align}[e8] \label{e8}
{{\left( 1-\frac{\epsilon}{5} \right)(n+K_1) + 2^{s_1} - 1 + 2^{s_1}} \choose 2^{s_1}} \leq & {{n+K_1 + 2^{s_1+1} - 1 } \choose 2^{s_1}} \nonumber \\
= & \frac{(n+K_1 + 2^{s_1+1} - 1 ) \hdots (n+K_1 + 2^{s_1})}{(2^{s_1})!} \nonumber \\
\leq & (n+K_1 + 2^{s_1+1} - 1)^{2^{s_1}}.
\end{align}
Now by \ref{e6}, \ref{e7} and \ref{e8} we get,
\begin{align*}
|\ao{H}^{(s)}(n+K_1)| & \leq 3 \sum_{i_1, \hdots, i_{2^{s_1}}} |\as{B}{s_1}(i_1)| \hdots |\as{B}{s_1}(i_{2^{s_1}})| \nonumber \\
& \leq \sum_{i_1, \hdots, i_{2^{s_1}}} |\as{B}{s_1-1}(I)|^{2^{s_1}} (\as{\lambda}{s_1} + \delta)^{\left( 1-\frac{\epsilon}{5} \right)(n+K_1) + 2^{s_1} - 1} \nonumber \\
& \leq (n+K_1 + 2^{s_1+1} - 1)^{2^{s_1}} |\as{B}{s_1-1}(I)|^{2^{s_1}} (\as{\lambda}{s_1} + \delta)^{\left( 1-\frac{\epsilon}{5} \right)(n+K_1) + 2^{s_1} - 1}
\end{align*}
By \ref{e5},
\begin{equation} \label{e9}
|\ao{B}(n)| \leq 3 K_1 (n+K_1 + 2^{s_1+1} - 1)^{2^{s_1}} |\as{B}{s_1-1}(I)|^{2^{s_1}} (\as{\lambda}{s_1} + \delta)^{\left( 1-\frac{\epsilon}{5} \right)(n+K_1) + 2^{s_1} - 1}
\end{equation}
Therefore,
\begin{align*}
\ao{\lambda} = & \lim_n |\ao{B}(n)|^{1/n}\\
\leq & \lim_n \left[ 3 K_1 (n+K_1 + 2^{s_1+1} - 1)^{2^{s_1}} |\as{B}{s_1-1}(I)|^{2^{s_1}} (\as{\lambda}{s_1} + \delta)^{\left( 1-\frac{\epsilon}{5} \right)(n+K_1) + 2^{s_1} - 1} \right]^{1/n} \\
\leq & \left( \as{\lambda}{s_1} + \delta \right)^{\left( 1 - \frac{\epsilon}{5} \right)}
\end{align*}
Since $\delta$ is arbitrary,
\begin{equation*}
\ao{\lambda} \leq \left(\as{\lambda}{s_1} \right)^{\left( 1 - \frac{\epsilon}{5} \right)}.
\end{equation*}
In the same way, still under the assumption \ref{e4}, and replacing $\omega$ by $\omega, \sigma^{s_1} \omega, \sigma^{s_1 + s_2} \omega, \sigma^{s_1 + s_2 + s_3} \omega, \hdots$, we get,
\begin{align*}
\ao{\lambda} & \leq \left(\as{\lambda}{s_1} \right)^{\left( 1 - \frac{\epsilon}{5} \right)} \\
\as{\lambda}{s_1} & \leq \left(\as{\lambda}{s_1 + s_2} \right)^{\left( 1 - \frac{\epsilon}{5} \right)} \\
\as{\lambda}{s_1 + s_2} & \leq \left(\as{\lambda}{s_1 + s_2 + s_3} \right)^{\left( 1 - \frac{\epsilon}{5} \right)} \\
& \vdots
\end{align*}
Thus for each $k \in \mathbb{N}$,
\begin{equation} \label{e10}
\ao{\lambda} \leq \left(\as{\lambda}{s_1 + \hdots + s_k} \right)^{\left( 1 - \frac{\epsilon}{5} \right)^k}.
\end{equation}
But the growth index $\lambda$ of a group with $8$ generators of order $2$ cannot exceed $9$. Since $k$ may be chosen arbitrarily large, it follows from \ref{e10} that $\ao{\lambda} = 1$.
If there exists an $\epsilon >0$ satisfying \ref{e4}, then $\ao{\lambda} = 1$. If not, then for all $\epsilon >0$ we have \ref{e2.1}. Thus by \ref{e3} and,
\[ \lim_{\epsilon \rightarrow 0} \, (1 - 2\epsilon)^{-1/2} (\epsilon / 3)^{-\epsilon} = 1 \]
we get $\ao{\lambda} = 1$ in all cases. Since $\ao{\lambda} = 1$, $\ao{G}$ has subexponential growth.

We know $G_{\omega} \subset \ao{G}$ and by \cite{Gri84a}, $G_{\omega}$ is of intermediate growth. Therefore $\ao{G}$ is of intermediate growth.
\end{proof}

\section{Growth bounds for $\bm{\ao{G}}$} \label{growthbounds}

\begin{prop}\label{lowerbound}
	Let $\omega \in \Omega_0 \cup \Omega_1$. Then for each $\epsilon >0$,
	\[ \gamma_{\ao{G}}(n) \succeq \exp{\left(\dfrac{n}{\log^{2+\epsilon}(n)}\right)}. \]
\end{prop}

\begin{proof}
	Let $\omega \in \Omega_0 \cup \Omega_1$. We may assume $\omega$ has infinitely many 0's and 2's. Then $b_\omega$ as a sequence of $P$'s and $I$'s contains both symbols infinitely often. By theorem 2 of \cite{Ers04} the group generated by elements $a, b_\omega, x$ has growth bounded below by $ \exp{\left(\dfrac{n}{\log^{2+\epsilon}(n)}\right)}$. Since $\ao{G}$ contains the elements $a, b_\omega, x$, we get the required result.
\end{proof}

\begingroup
\setcounter{tmp}{\value{theorem}}
\setcounter{theorem}{0} 
\renewcommand\thetheorem{\ref{upperbound}$^\prime$}
\begin{theorem}\label{upperboundomega1}
	Let $\omega \in \Omega_1$ with infinitely many $i,j$'s. Suppose there exists an integer $M$ such that for any $k \geq 1$, $\{ \omega_k, \omega_{k+1}, \hdots , \omega_{k+M-1} \}$ contains both $i,j$. Then,
	\[ \gamma_{\ao{G}}(n) \preceq \exp{\left(\dfrac{n \log{(\log{(n)}}}{\log{(n)}}\right)}. \]
\end{theorem}
\endgroup
\setcounter{theorem}{\thetmp} 

\begin{proof}
	Since $\omega \in \Omega_1$, there is an $N$ such that $\sigma^N$ contains only $i,j$'s. Then by lemma \ref{4}, $\as{G}{N} = G_{\sigma^N \omega}$. Then by theorem 3 of \cite{Ers04},
	\[ \gamma_{\as{G}{n}}(n) \preceq \exp{\left(\dfrac{n \log{(\log{(n)}}}{\log{(n)}}\right)}. \]
	Therefore,
	\begin{align*}
	\gamma_{\ao{G}}(n) & \approx \left(\gamma_{\as{G}{n}}(n)\right)^{2^N} \\
	 & \preceq \left(\exp{\left(\dfrac{n \log{(\log{(n)}}}{\log{(n)}}\right)}\right)^{2^N} \\
	 & \approx \exp{\left(\dfrac{n \log{(\log{(n)}}}{\log{(n)}}\right)}.
	\end{align*}
\end{proof}

\begingroup
\setcounter{tmp}{\value{theorem}}
\setcounter{theorem}{0} 
\renewcommand\thetheorem{\ref{upperbound}$^{\prime \prime}$}
\begin{theorem}\label{upperboundomega0}
	Let $\omega \in \Omega_0$. Suppose there exists an integer $M$ such that for any $i \geq 1$, $\{ \omega_i, \omega_{i+1}, \hdots , \omega_{i+M-1} \}$ contains all three symbols. Then,
	\[ \gamma_{\ao{G}}(n) \leq \exp{\left(\dfrac{n \log{(\log{(n)}}}{\log{(n)}}\right)}. \]
\end{theorem}
\endgroup
\setcounter{theorem}{\thetmp} 

\begin{proof}
	The proof follows similarly as of the proof of theorem 3 of \cite{Ers04} by replacing lemma 6.2 (1) of \cite{Ers04} by lemma \ref{11}.
\end{proof}

Theorem \ref{upperboundomega1} together with theorem \ref{upperboundomega0} implies theorem \ref{upperbound}.


\begin{thebibliography}{4}
  
  
  
  
\bibitem{Bar98}
Bartholdi, Laurent 
\textbf{The  growth  of  Grigorchuk’s  torsion  group}.
\emph{Internat.  Math.  Res.  Notices (1998), no. 20}, 1049-1054.

\bibitem{Bar03}
Bartholdi, Laurent 
\textbf{A Wilson group of non-uniformly exponential growth}.
\emph{C. R. Math.Acad. Sci. Paris, 336 (2003), no. 7}, 549-554.

\bibitem{BGN15}
Benli, Mustafa G.; Grigorchuk, Rostislav I.; Nagnibeda, Tatiana
\textbf{Universal groups of intermediate growth and their invariant random subgroups}.
\emph{Funct. Anal. Appl. 49 (2015), no. 3}, 159-174.

\bibitem{BE14}
Bartholdi, Laurent; Erschler, Anna
\textbf{Groups of given intermediate word growth}.
\emph{Annales de l'Institut Fourier, Volume 64 (2014) no. 5}, 2003-2036.

\bibitem{BG00}
 Bartholdi, Laurent; Grigorchuk, Rostislav I. 
 \textbf{On the spectrum of Hecke type operators related to some fractal groups}.
 \emph{Tr. Mat. Inst. Steklova} 231 (2000), \emph{Din. Sist., Avtom. i Beskon. Gruppy}, 5-45; translation in \emph{Proc. Steklov Inst. Math. 2000, no. 4(231)}, 1-41.
 
\bibitem{BG02}
 Bartholdi, Laurent; Grigorchuk, Rostislav I.
 \textbf{On parabolic subgroups and Hecke algebras of some fractal groups}.
 \emph{Serdica Math. J. 28 (2002), no. 1}, 47-90.
  
\bibitem{Ers04}
 Erschler, Anna
 \textbf{Boundary behavior for groups of subexponential growth}.
 \emph{Annals of Mathematics, 160, (2004), no 3}, 1183-1210.
 
\bibitem{EZ18}
Erschler, Anna; Zheng, Tianyi
\textbf{Growth of periodic Grigorchuk groups}.
\emph{(2018)}, arXiv preprint arXiv:1802.09077 [math.GR]. 

\bibitem{Fra17}
 Francoeur, Dominik
\textbf{On the subexponential growth of groups acting on rooted trees}.
\emph{(2017)}, arXiv preprint arXiv:1702.08047 [math.GR].

\bibitem{Gri80}
 Grigorchuk, Rostislav I.
 \textbf{On Burnside's problem on periodic groups}.
 \emph{Funct. Anal. Appl. 14 (1980), no. 1}, 41-43.

\bibitem{Gri83}
 Grigorchuk, Rostislav I.
 \textbf{On the Milnor problem of group growth}.
 \emph{Soviet Math. Dokl. 28 (1983), no. 1}, 23-26.

\bibitem{Gri84a}
 Grigorchuk, Rostislav I.
 \textbf{Degrees of growth of finitely generated groups and the theory of invariant means}.
 \emph{Izv. Akad. Nauk SSSR Ser. Mat. 48 (1984), no. 5}, 939-985.
 
\bibitem{Gri84b}
 Grigorchuk, Rostislav I.
 \textbf{Construction of p-groups of intermediate growth that have a continuum of factor-groups}.
 \emph{Algebra i Logika 23 (1984), no. 4}, 383-394.
 
\bibitem{Gri85}
 Grigorchuk, Rostislav I.
 \textbf{Degrees of growth of p-groups and torsion-free groups}.
 \emph{Mat. Sb. (N.S.) 126 (168) (1985), no. 2}, 194-214.
 
\bibitem{Gri91}
Grigorchuk, Rostislav I.
\textbf{On growth in group theory}. 
\emph{Proceedings of the International Congressof  Mathematicians Vol I (August 21-29, 1990), Kyoto, Japan, Math.  Soc.  Japan,  (1991)}, 325-338.

\bibitem{KP01}
Kassabov, Martin; Pak, Igor
\textbf{Groups of Oscillating Intermediate Growth}.
\emph{Annals of Mathematics, 177 (2013), no. 3}, 1113-1145.

\bibitem{Mil68}
Milnor, John
\textbf{Problem 5603}.
\emph{Amer. Math. Monthly 75 (1968), no. 6}, 685-686.

\bibitem{MP01}
Muchnik, Roman; Pak, Igor
\textbf{On growth of Grigorchuk groups}.
\emph{Internat. J. Algebra Comput. 11 (2001), no. 1}, 1-17.

\bibitem{Nek18}
Nekrashevych, Volodymyr
\textbf{Palindromic subshifts and simple periodic groups of intermediate growth}.
\emph{Annals of Mathematics, 187, (2018), no 3}, 667-719.

\bibitem{Sch55}
Schwarz, A. I.
\textbf{A volume invariant of covering}.
\emph{Dokl. Akad. Nauj SSSR (1955), no. 105}, 32-34.
	
	

	
	
\end{thebibliography}
\end{document}